\documentclass[12pt]{amsart}
\usepackage{stmaryrd} 
\usepackage{hyperref} 
 \theoremstyle{plain}
 \newtheorem*{teorema}{Theorem}
 \newtheorem{theorem}{Theorem}
  \newtheorem{lemma}[theorem]{Lemma}
 \newtheorem{proposition}[theorem]{Proposition}
 
 \theoremstyle{definition}
 \newtheorem{remark}[theorem]{\bf Remark}

 \newcommand{\N}{\mathbb N}
 \newcommand{\T}{\mathbb T}
 \newcommand{\Z}{\mathbb Z}
 \newcommand{\R}{\mathbb R}
 \newcommand{\mS}{\mathbb S}
 \newcommand{\cA}{\mathcal A}
 \newcommand{\tcA}{\tilde\cA}
 \newcommand{\cC}{\mathcal C}

 \newcommand{\cL}{\mathcal L}
 \newcommand{\cM}{\mathcal M}
 \newcommand{\cP}{\mathcal P}
 \newcommand{\tcM}{\widetilde{\mathcal M}}
 \newcommand{\cS}{\mathcal S}

 \newcommand{\ub}{\bar{u}}
 \newcommand{\xb}{\bar{x}}
 \newcommand{\yb}{\bar{y}}
 \newcommand{\af}{\alpha}
 \newcommand{\be}{\beta}
 
 \newcommand{\ga}{\gamma}
 \newcommand{\Ga}{\Gamma}
 \newcommand{\de}{\delta}
 
 \newcommand{\lam}{\lambda}
 
 \newcommand{\om}{\omega}
\newcommand{\afb}{\bar{\alpha}}
 \newcommand{\beb}{\bar{\beta}}
 \newcommand{\gab}{\bar{\gamma}}
\newcommand{\lb}{\left\llbracket}
\newcommand{\rb}{\right\rrbracket}

 \newcommand{\mL}{\mathbb L}
 \newcommand{\mH}{\mathbb H}
 \newcommand{\dga}{\dot{\ga}}
 \newcommand{\dbe}{\dot{\be}}
 \newcommand{\sop}{\operatorname{supp}}

\begin{document}
 \title[Exponential Convergence]{Hyperbolicity and exponential convergence
of the Lax Oleinik semigroup for Time Periodic Lagrangians} 
 \author[H. S\'anchez-Morgado]{H\'ector S\'anchez-Morgado}
\address{ Instituto de Matem\'aticas, UNAM. Ciudad Universitaria C. P. 
          04510, Cd. de M\'exico, M\'exico.}
 \email{hector@math.unam.mx}

 \maketitle
 
 \begin{abstract}
In this note we prove exponential convergence to time-periodic states of 
the solutions of time-periodic Hamilton-Jacobi equations on the torus,
assuming that the Aubry set is the union of a finite number of
hyperbolic periodic orbits of the Euler Lagrange flow.
The period of limiting solutions is the least common multiple of the periods
of the orbits in the Aubry set. This extends a result that we obtained
in \cite{IS} for the autonomous case.
 \end{abstract}

 \section{Introduction}
 
 Let $M$ be a closed connected manifold, $TM$ its tangent bundle. Let
 $L:TM\times \R\to \R$ be a $C^k$ $k\ge2$ Lagrangian. We will assume for
 the Lagrangian the hypothesis of Mather's seminal paper~\cite{M}.
 \begin{enumerate}
 \item 
 {\em Convexity}. The Lagrangian $L$ restricted to $T_xM$, in linear 
 coordinates has positive definite Hessian.
\item 
{\em Superlinearity}. For some Riemannian metric we have 
 $$
 \lim_{|v|\to \infty}\frac{L(x,v,t)}{|v|}=\infty,
 $$
 uniformly on $x$ and $t$.
\item 
 {\em Periodicity}. The Lagrangian is periodic in time, i.e.
 $$L(x,v,t+1)=L(x,v,t),$$
 for all $x,v,t$.
\item 
 {\em Completeness}. The Euler Lagrange flow $\phi_{t}$ associated to the
   Lagrangian is complete.
 \end{enumerate}
Let $H:T^*M\times\R\to \R$ be the Hamiltonian associated to the Lagrangian:
 \begin{equation}
   \label{eq:hamiltonian}
   H(x,p,t)=\max_{v\in T_xM}pv-L(x,v,t).
 \end{equation}
For $c\in\R$ consider the Hamilton Jacobi equation 
\begin{equation}\label{hj}
u_t+H(x,d_xu,t)=c.
\end{equation}
It is know (\cite{CIS}, \cite{B1}) that there is only one value $c=c(L)$, the so
called critical value, such that \eqref{hj} has a time periodic
viscosity solution. 
In this article we prove
\begin{teorema}
Assume $M=\T^d$ and the Aubry set is the union of a
finite number of hyperbolic periodic orbits 
of  the Euler Lagrange flow and let $N$ be the least common multiple of
their periods. There is $\mu>0$ such that for any viscosity solution
$v:\T^d\times[s,\infty[\to\R$ of \eqref{hj} with $c=c(L)$,  
there is an $N$-periodic viscosity solution $u:\T^d\times\R\to\R$ 
of \eqref{hj} and $K>0$ such that for any $t>0$ 
\[|v(x,t)-u(x,t)|\le K e^{-\mu t}\]
\end{teorema}
In \cite{IS} we proved a similar result in the autonomous case under
the assumption that the Aubry set consists in a finite number of 
hyperbolic critical points of the Euler Lagrange flow and claimed 
wrongly, as observed in \cite{WJ}, that it also holds when one changes some
critical points by periodic orbits.
Part of the present note follows same ideas as in the autonomus case, but
it was necessary to provide proofs of some statements in that part,
which are the contents of Lema \ref{single} and Proposition \ref{cadena}.

\section{Preliminaries on weak KAM theory}
\label{sec:preliminar}
Define the action of an absolutely  continuous curve  $\ga:[a,b]\to M$ as
 $$A_L(\ga):=\int_{a}^{b}L(\ga(\tau),\dot\ga(\tau),\tau) d\tau $$
A curve $\ga:[a,b]\to M$ is {\em closed } if
$\ga(a)=\ga(b)$ and $b-a\in\Z$.  The critical value can be defined as
\[c(L):= \min \{k\in\R:\forall\ga\text{ closed }A_{L+k}(\ga) \ge 0\}\]
Let ${\cP}(L)$ be the set of probabilities on the Borel $\sigma$-algebra 
of $TM\times\mS^1$ that have compact support and are invariant under
the Euler-Lagrange flow. Then
\[c(L)= -\min \{\int Ld\mu : \mu\in\cP(L)\}.\]
The Mather set is defined as 
\[\tcM:=
\overline{\bigcup\{\sop\mu:\mu\in\cP(L),\int Ld\mu=-c(L)\}}.\]
 For $a\le b$, $x,y\in M$ let $\cC(x,a,y,b)$ be the set 
 of absolutely continuous curves $\ga:[a,b]\to M$ with $\ga(a)=x$ and
 $\ga(b)=y$. 
For $a\le b$ define $F_{a,b}:M\times M\to\R$ by  
  $$F_{a,b}(x,y):= \min\{A_L(\ga):\ga\in\cC(x,a,y,b)\}.$$
Define $\cL_t:C(M,\R)\to C(M,\R)$ by 
\[\cL_tu(x)=\inf_{y\in M}u(y)+F_{0,t}(y,x).\]
Then $v(x,t)=\cL_tu(x)$ is the viscosity solution to \eqref{hj} for $c=0$
with $v(x,0)=u(x)$.
$(\cL_n)_{n\in\N}$ is a semi-group known as the Lax-Oleinik semi-group.

For $t\in\R$ let $[t]$ be the corresponding point in
${\mS}^1$ and $\lb t\rb$ be the integer part of $t$.
 Define the {\em action functional}
$\Phi:M\times{\mS}^1\times M\times{\mS}^1\to\R$ by  
 \[\Phi(x,[s],y,[t]):=\min\{F_{a,b}(x,y)+c(L)(b-a):[a]=[s], [b]=[t]\},\]
 and the {\em Peierls barrier} $h:M\times{\mS}^1\times M\times{\mS}^1\to\R$ by
\begin{equation}
  \label{eq:barrier}
 h(x,[s],y,[t]):=\liminf_{\lb b-a\rb\to\infty}
\bigl(F_{a,b}(x,y)+c(L)(b-a)\bigr)_{[a]=[s], [b]=[t]} 
\end{equation}
 We have $-\infty<\Phi\le h<\infty$.

The critical value is the unique number $c$ such that \eqref{hj}
has viscosity solutions $u:M\times\mS^1\to\R$. 
In fact \cite{CIS}, for any $(p,[s])$ the functions $(x,[t])\mapsto h(p,[s],x,[t])$ 
and $(x,[t])\mapsto-h(x,[t],p,[s])$ are respectively 
{\em backward} and {\em forward} viscosity solutions of \eqref{hj}. 
Set $c=c(L)$ and let $\cS^- (\cS^+)$ be the set of {\em backward}
({\em forward}) viscosity solutions of \eqref{hj}.

The Lagrangian is called regular if the $\liminf$ in
\eqref{eq:barrier} is a $\lim$ and in that case, for each $s,t\in\R$
the convergence of the sequence $(F_{a,b})_{[a]=[s], [b]=[t]}$ is uniform. 
The Lagrangian is regular if and only if the Lax-Oleinik semi-group
converges \cite{B1}, so that for each $u\in C(M,\R)$ there exists
$\ub\in C(M\times\mS^1,\R)$ such that
\[\lim_{k\to \infty}\cL_{t+k}u(x)+c(t+k)=\ub(x,[t])\]
and in fact
\[\ub(x,[t])=\inf_{y\in M}u(y)+h(y,[0],x,[t]).\]
A subsolution of \eqref{hj} always means a viscosity subsolution.
A curve $\ga:I\to M$ {\em calibrates} a subsolution
$u:M\times\mS^1\to\R$ of \eqref{hj} if 
\[u(\ga(b),[b])-u(\ga(a),[a])=A_{L+c}(\ga|[a,b])\]
for any $[a,b]\subset I$. If $u\in\cS^- (\cS^+)$, for any 
$(x,[s])\in M\times\mS^1$ there is $\ga:]-\infty,s]\to M$
($\ga:[s,\infty[\to M$) that calibrates $u$ and $\ga(s)=x$.
  
A pair $(u_-,u_+)\in \cS^-\times\cS^+$
is called {\em conjugated} if $u_-=u_+$ on $\cM=\pi(\tcM)$.
For such a pair $(u_-,u_+)$, we define $I(u_-,u_+)$ as the set where
$u_-$ and $u_+$ coincide. If $(x,[s])\in I(u_-,u_+)$ then 
$u_\pm$ is differentiable at $(x,[s])$ and
$d_xu_-(x,[s])=d_xu_+(x,[s])$. Let
\[I^*(u_-,u_+)=\{(x,d_xu_\pm(x,[s]),[s]): \,(x,[s])\in I(u_-,u_+)\}.\]
We may define the {\em Aubry set} either as the set (\cite{B}) 
\[\cA^*:=\bigcap\{I^*(u_-,u_+):
(u_-,u_+)\text{ conjugated}\}\subset T^*M\times\mS^1\]
or as its pre-image under the Legendre transformation (\cite{F})
\[\tcA:=\{(x,H_p(x,p,t),[t]):(x,p,[t])\in\cA^*\}.\]
 The projection of either Aubry set in $M\times\mS^1$ is
 \[\cA=\{\,(x,[t])\in M\times\mS^1:h(x,[t],x,[t])=0\}.\]
An important tool for the proof of our result is the existence of 
strict $C^k$ critical subsolutions in our setting, that extends the result  
of Bernard \cite{B} for the autonomous case and its written out in the 
thesis of E. Guerra \cite{G}.
\begin{theorem}\label{subcrit}
Assume the Aubry set $\tcA$ is the union of a finite number of
hyperbolic periodic orbits of the the Euler Lagrange flow, then there is
a $C^k$ subsolution $f$ of \eqref{hj} such that 
\[f_t+H(x,d_xf(x,[t]),t)< c\]
for any $(x,[t])\notin\cA$.
\end{theorem}

\section{Reduction to a regular Lagrangian}
\label{sec:reduction}

Let $M=\T^d$ and assume the Aubry set $\tcA$ is the union of the
hyperbolic periodic orbits 
\[\Ga_i(t)=\phi_t(x_i,v_i,[0])=(\ga_i(t),\dga_i(t),[t]) \quad i\in[1,m]\]
of the Euler Lagrange flow with periods $N_i, i\in[1,m]$.
In this case the projected Aubry and Mather set coincide.
Let $N$ be the least common multiple of $N_1,\ldots,N_m$. Define
\begin{align}\label{Levanta}
P_N:\T^d\times\R^d\times\mS^1 & \to \T^d\times\R^d\times\mS^1\\
                  (x,v,[t])  & \mapsto (x,\frac vN,[Nt]),\nonumber
\end{align}
and the Lagrangian $L_N=L\circ P_N$. The corresponding Hamiltonian is
given by 
\[H_N(x,p,t)=H(x,Np,Nt).\]
For a curve $\ga:[a,b]\to\T^d$ define
$\ga^N:[a/N,b/N]\to\T^d, t\mapsto\ga(Nt)$, then 
$NA_{L_N}(\ga^N)=A_L(\ga)$.
A curve $\ga$ is an extremal (minmizer)  of $L$ if and
only if the curve $\ga^N$ is an extremal (minimizer) of $L_N$.

Let $\ga^N_{i,j}(t)=\ga_i(j+Nt),j\in[1,N_i],i\in[1,m]$.
According to sections 3, 5 of \cite{B1}, the Aubry set of $L_N$ is the
union of the hyperbolic 1-periodic orbits 
$\Ga_{i,j}^N(t)=(\ga^N_{i,j}(t)\dga^N_{i,j}(t),[t])$ and $L_N$ is regular. 
Observe that a function $u:\T^d\times\R\to\R$ is a viscosity solution of 
\eqref{hj} if and only if $w(x,t)=\frac 1Nv(x,Nt)$ is a viscosity
solution of
\[ w_t+H_N(x,d_xw,t)=0.\]
Thus our main Theorem is reduced to the case in which the Lagrangian is regular
and the Aubry set is the union of finite number of hyperbolic 1-periodic orbits.
 
Let $f:\T^{d+1}\to\R$ be a $C^k$ subsolution of \eqref{hj} strict
outside $\cA$. Consider the Lagrangian
\[\mL(x,v,t)=L(x,v,t)-d_xf(x,[t])v-f_t(x,[t])+c\]
with Hamiltonian $\mH(x,p,t)=H(x,p+d_xf,[t])+f_t(x,[t])-c$. 

If $\af\in\cC(x,s,y,t)$, $A_{\mL}(\af)=A_{L+c}(\af)+f(x,[s])-f(y,[t])$.
Thus $L$ and $\mL$  have the same Euler Lagrange flow and projected
Aubry set. Moreover,
\begin{equation}\label{mL} 
\forall (x,v,t)\; \mL(x,v,t)\ge 0,\;
\tilde\cA=\{(x,v,t):\mL(x,v,t)=0\}
\end{equation}
and $u$ is a viscosity solution of \eqref{hj}
if and only if $u-f$ is a viscosity solution of 
\[w_t+\mH(x,d_xw,t)=0.\]
We can therefore assume that $c=0$, $L$ is regular and has the property 
\eqref{mL} of $\mL$, and the Aubry set is the union of hyperbolic
1-periodic orbits $\Ga_i, i\in[1,m]$ of the Euler Lagrange flow.

Thus, for any $u\in C(\T^d,\R)$ we have  
\begin{equation}\label{eq:converge}
\lim_{k\to\infty}\cL_{\tau+k}u(x)=\ub(x,[\tau])
:=\min_{y\in\T^d}u(y)+h(y,[0],x,[\tau]).
\end{equation}
We get our result by proving that the convergence in \eqref{eq:converge} 
is exponentially fast.
The following Lemmas will be helpful
\begin{lemma}\label{union}
Let $W=\bigcup\limits_{i=1}^mW_i$ be a neighborhood of the Aubry set
in $\T^d\times\R^d\times\mS^1$.
Then, there exist $C>0$ such that if $\ga:[0,t]\to\T^d$,
$t>1$ is a minimizer, then the time that $(\ga,\dga,[\cdot])$
remains outside $W$ is less than $C$.
\end{lemma}
\begin{proof}
Since the Lagrangian is non-negative and it vanishes only on the Aubry
set, outside the neighborhood $W$ it is bounded from below by $\eta>0$. 
For $t>1$ the action of minimizers $\ga:[0,t]\to\T^d$ is bounded from
above independently of $t$. The Lemma follows. 
\end{proof}
\begin{lemma}\label{single}
Let $V=\bigcup\limits_{i=1}^mV_i$ be a neighborhood of the Aubry set
in $\T^d\times\R^d\times\mS^1$.
There is $N=N(V)\in\N$ such that if
$\ga:[0,t]\to\T^d, t>1$ is minimizer, then $(\ga,\dga,[\cdot])$ stays in one
$V_i$ during an interval larger than $\dfrac{t}{N}-1$.    
\end{lemma}
\begin{proof}
Let $\de\in]0,1[$ be such that the $\de$-neighborhood of $\Ga_i$ is
contained in $V_i$ for $i\in[1,m]$. 
Let $W_i$ be the $\dfrac\de 2$-neighborhood of $\Ga_i$ and
apply Lemma \ref{union} to $W=\bigcup\limits_{i=1}^mW_i$.
Since the velocity of any minimizer $\ga:[0,t]\to\T^d, t>1$  
is bounded by a constant $\ge 1$, the number of times it can go from $W$ to
the complement of $V$ is bounded by some $N\in\N$. The Lemma follows. 
\end{proof}
Let $\lambda_{i,j},j=1,\ldots,d$ be the positive Lyapunov exponents of
$\ga_i$, set $\lambda<\min\limits_{i,j}\lambda_{ij}$ and write
$\phi_t(x,v,[0])=\psi_t(x,v)$. There is a splitting
$\R^{2d}=E_i^-\oplus E_i^+$, invariant under $D_i=d\psi_1(x_i,v_i)$,
and a norm $\|\cdot\|$ such that $\|D_i^{\pm 1}|E_i^\mp\|\le e^{-\lam}$. 
\begin{proposition}[\cite{BV}, \cite{Be}]\label{conjugacy}
There are $\af,\rho\in (0,1)$, neighborhoods $A_i$ of $(x_i,v_i)$ in
$\T^d\times\R^d$, and $\af$-H\"older maps $g_i:A_i\to B_{2\rho}(0)\subset\R^{2d}$ 
with $\af$-H\"older inverse such that 
\[D_i\circ g_i=g_i\circ\psi_1.\]
\end{proposition}
Set
\begin{equation}
  \label{eq:vecindad}
U_i=g_i^{-1}(B_\rho),\quad
V_i=\bigcup_{s\in[0,1]}(\psi_s(U_i)\cap\psi_{s-1}(U_i))\times[s],
\quad V=\bigcup_{i=1}^mV_i.
\end{equation}

\section{Proof of the result}
\label{sec:proof}
For brevity, for $x\in\T^d$ we write $\bar x=(x,[0])$, 
and for $\be:I\to\T^d$ we write $\beb(t)=(\be(t),[t])$,  
$d\be(t)=(\be(t),\dbe(t))$, $B(t)=(\be(t),\dbe(t),[t])$.

For each $z\in\T^{d+1}$ let $y_z$ be such that 
\[\ub(z)=u(y_z)+h(\yb_z,z)\] 
In particular let $y_j=y_{\xb_j}$. 
Since $z\mapsto -h(z,\xb_j)\in\cS^+$, there is a semistatic curve
$\be^j:[0,\infty[\to\T^d$ such that  $\be^j(0)=y_j$ and
\[ A_L(\be^j|[0,t])=h(\yb_j,\xb_j)-h(\beb^j(t),\xb_j),\quad t>0.\]
We can take $\rho>0$ given in Proposition \ref{conjugacy} such that
\[\{B^j(0):j\in[1,m]\}\cap V-\tcA=\emptyset\]

Fix $j\in[1,m]$, set $\be_0=\be^j$ and let $\Ga_{k_1}$ be the
$\om$-limit of $B_0=B^j$. 
If $k_1=j$ we stop, otherwise we observe that by the regularity of $L$
\[h(\beb_0(t),\xb_j)=h(\beb_0(t),\xb_{k_1})+h(\xb_{k_1},\xb_j),\quad t\ge 0 \]
and then
\[A_L(\be_0|[0,t])=h(\yb_j,\xb_{k_1})-h(\beb_0(t),\xb_{k_1}),\quad
t>0.\]

\begin{proposition}\label{cadena}
If $k_1\ne j$, there are $k_1,\ldots,k_l=j$ all different  
and semistatic curves $\be_r:\R\to\T^d$,
$r< l$ such that $\Ga_{k_r}$ and $\Ga_{k_{r+1}}$ are the
$\af$ and $\om$ limits of $B_r$,
\[h(\xb_{k_1},\xb_j)=\sum_{r=1}^{l-1}h(\xb_{k_r},\xb_{k_{r+1}}),\]
\[A_L(\be_r|_{[t,s]})=h(\xb_{k_r},\beb_r(s))-h(\xb_{k_r},\beb_r(t)),\quad
t\le s.\]
\end{proposition}
\begin{proof}
There is a neighborhood $U'_i$of $\gab_i$ where $z\mapsto h(\xb_i,z)$
is $C^k$ and the local weak unstable manifold of $\Ga^*_i$ is the graph
of $d_xh(\xb_i,x,[t])$. Let $U_i$ be a neighborhood of $\gab_i$  
with compact $\bar U_i\subset U'_i$.
Let $\rho_n:[0,n]\to\T^d$ be a curve joining $x_{k_1}$ to
$x_j$ such that
\[A_L(\rho_n)=F_{0,n}(x_i,x_j).\]
Let $t_n\in[0,n]$ be the first exit time of $\gab_n(t)$ out of $U_i$, and
$\gab_n(t_n)$ be the first point
of intersection with $\partial U_{k_1}$. As $n$ goes to infinity,
$t_n$ and $n-t_n$ tend to infinity. This follows from the fact that
$\dot\rho_n(0)$ has to tend to $\dga_{k_1}(0)$, and
$\dot\rho_n(n)$ has to tend to $\dga_j(0)$. To justify this, 
consider $v$ a limit point of $\dot\rho_n(0)$, and $\ga:\R\to\T^d$
the solution to the Euler-Lagrange equation such that
$\ga(0)=x_i,\dga(t)=v$. From the fact that
\[F_{0,n}(x_{k_1},x_j)-F_{1,n}(\rho_n(1),x_j)=A_L(\rho_n|_{[0,1]})\]
and the regularity of $L$, taking limit $n\to \infty$ it follows  
\[h(\xb_{k_1},\xb_j)-h(\gab(1),\xb_j)=A_L(\ga|_{[0,1]}).\]
Since $\ga_{k_1}(-1)=x_{k_1}$ and $L=0$ on $\tcA$
\[h(\gab_{k_1}(-1),\xb_j)-h(\gab(1),\xb_j)=
A_L(\ga_{k_1}|_{[-1,0]})+A_L(\ga_{[0,1]})\]
so that the curve obtained by gluing $\ga_{k_1}|_{[-1,0]}$ with
$\ga|_{[0,1]}$ minimizes the action between its endpoints. In
particular, it has to be differentiable, thus $v=\dga(0)=\dga_{k_1}(0)$.
Define $\af_n:[-\lb t_n\rb,n-\lb t_n\rb]\to\T^d$ by 
$\af_n(t)=\rho_n(t+\lb t_n\rb)$ and let $(y,w,\tau)$ be a cluster point
of $(d\rho_n(t_n),t_n-\lb t_n\rb)$. 
Then there is a sequence $(\af_{n_l})$ converging uniformly on compact
intervals to the solution $\be_1:\R\to\T^d$ of the Euler-Lagrange equation
such that $d\be_1(\tau)=(y,w)$. Since for any $t\le s$ we have
\begin{align*}
F_{-\lb t_n\rb,s}(x_{k_1},\af_n(s))-F_{-\lb t_n\rb,t}(\xb_{k_1},\af_n(t))
&=A_L(\af_n|_{[t,s]}),\\ 
F_{-\lb t_n\rb,t} (x_{k_1},\af_n(t))+F_{t,n-\lb t_n\rb}(\af_n(t),x_j)
&=F_{0,n}(x_{k_1},x_j),
\end{align*}
from the uniform convergence of $F_{a,b}$ when $\lb b-a\rb\to\infty$, 
we obtain for any $t\le s$  
\begin{align*}
  h(\xb_{k_1},\beb_1(s))-h(\xb_{k_1},\beb_1(t))&=A_L(\be_1|_{[t,s]})\\
h(\xb_{k_1},\beb_1(t))+h(\beb_1(t),\xb_j)&= h(\xb_{k_1},\xb_j).
\end{align*}
Since $\afb_n([-\lb t_n\rb,t_n-\lb t_n\rb])\subset\bar U_{k_1}$ we
have that $\Ga_{k_1}$ is the $\af$-limit of $B_1$ and
let $\Ga_{k_2}$ be its $\om$-limit. 
 If $k_2=j$ we stop, otherwise we observe that
\begin{align*}
  h(\xb_{k_1},\xb_{k_2})+h(\xb_{k_2},\xb_j)&= h(\xb_{k_1},\xb_j),\\
h(\xb_{k_1},\beb_1(t))+h(\beb_1(t),\xb_{k_2})&= h(\xb_{k_1},\xb_{k_2}).
\end{align*}
Therefore $k_1\ne k_2$. We proceed in the same way to find a solution
to the Euler-Lagrange equation  $\be_2:\R\to\T^d$ such that
$\Ga_{k_2}$ is the $\af$-limit of  $B_2$ and
for any $t,s\in\R$, $t<s$   
\begin{align*}
  h(\xb_{k_2},\beb_2(s))-h(\xb_{k_2},\beb_2(t))&=A_L(\be_2|_{[t,s]})\\
h(\xb_{k_2},\beb_2(t))+h(\beb_2(t),\xb_j)&= h(\xb_{k_2},\xb_j).
\end{align*}
Let $\Ga_{k_3}$ be the $\om$-limit of $B_2$.
If $k_3=j$ we stop, otherwise we observe that
\begin{align*}
  h(\xb_{k_2},\xb_{k_3})+h(\xb_{k_3},\xb_j)&= h(\xb_{k_2},\xb_j),\\
h(\xb_{k_2},\beb_2(t))+h(\beb_2(t),\xb_{k_3})&= h(\xb_{k_2},\xb_{k_3}).
\end{align*}
Therefore $k_2\ne k_3$. Since
\[h(\xb_{k_1},\xb_{k_2})+h(\xb_{k_2},\xb_{k_3})= h(\xb_{k_1},\xb_{k_3}),\]
we have that $k_1\ne k_3$. We continue until we get $k_r=j$.
\end{proof}
Let $V$ be given by \eqref{eq:vecindad}. 
There is $T>0$ such that for $t\ge T$, 
we have $B_r(t), B_r(-t)\in V$ and then
\begin{eqnarray*}
  d(d\be_r(t),d\ga_{k_{r+1}}(t))\le C_1e^{-\lam t},
&&t>T\\
d(d\be_r(t),d\ga_{k_r}(t))\le C_1e^{\lam t},
&&t<-T.
\end{eqnarray*}
For $(x,\tau)\in\T^d\times[0,1]$, we consider a directed graph with
vertices at the points $\xb_1,\cdots,\xb_m$ in the Aubry set, and a
directed segment from $\xb_j$ to $\xb_k$ if and only if
\[h(\xb_k,x,[\tau])=h(\xb_k,\xb_j)+h(\xb_j,x,[\tau]).\]
We call a point $\xb_k$ a root of this graph if there is no segment
arriving to this point and this means that for $j\ne k$
\[h(\xb_k,x,[\tau])<h(\xb_k,\xb_j)+h(\xb_j,x,[\tau]).\]
Notice that the graph contains no cycles, and so each point $\xb_k$
belongs to a branch starting at a root.  
If $\ub(x,[\tau])=\ub(\xb_k)+h(\xb_k,x,[\tau])$ and there is a
segment from $x_j$ to $x_k$, then
\begin{align*}
 \ub(x,[\tau])\le\ub(\xb_j)+h(\xb_j,x,[\tau])
&\le\ub(\xb_k)+h(\xb_k,\xb_j)+h(\xb_j,x,[\tau])\\
&=\ub(\xb_k)+h(\xb_k,x,[\tau])=\ub(x,[\tau]).
\end{align*}
Take $k\in[1, m]$ such that
\[\ub(x,[\tau])=\ub(\xb_k)+h(\xb_k,x,[\tau])\]
and let $\xb_j$ be the root of a branch containing $\xb_k$.
Since $z\mapsto h(\xb_j,z)\in\cS^-$, 
there is a semistatic curve $\be_{x,\tau}:]-\infty,\tau]\to\T^d$ with
$\be_{x,\tau}(\tau)=x$ such that
\[A_L(\be_{x,\tau}|[t,\tau])=
h(\xb_j,x,[\tau])-h(\xb_j,\beb_{x,\tau}(t)),\quad t<\tau.\] 
Let $\Ga_i$ be the $\af$-limit of $B_{x,\tau}$.
If $(x,[\tau])=\gab_j(\tau)$ then $\be_{x,\tau}=\ga_i=\ga_j$.
By the regularity of $L$
\[h(\xb_j,x,[\tau])=h(\xb_j,\xb_i)+h(\xb_i,x,[\tau]),\] 
since $\xb_j$ is a root, $i=j$.

From Lemma \ref{union} there is $\bar T\ge T$ such that 
$B_{x,\tau}(t)\in V_j$ for $t\le-\bar T$ 
\[d(d\be_{x,\tau}(t),d\ga_j(t))\le C_1e^{\lam t},\]
unless  $B_{x,\tau}$ is part of an orbit of the Euler-Lagrange 
flow with $\om$-limit some $\Ga_i$ and $B_{x,\tau}(\tau)\in V_i$.
If the last situation occurs let 
$M=\max\{q\in\N:B_{x,\tau}([-q,\tau])\subset V_i\}$ so that
$d(d\be_{x,\tau}(t),d\ga_i(t))\le C_2e^{-\lam(M+t)}.$
Otherwise take $M=0$.

For $k\in\N$ let $n+1=\lb\dfrac k{2(l+1)}\rb$ and define
the curve $\af_k:[0,\tau+k]\to\T^d$ by
\[\af_k(s)=
\begin{cases}
 \be_0(s) & s\in[0,n]\\
c_r(s-r(2n+1)+n+1)& s\in[r(2n+1)-n-1,r(2n+1)-n], r\le l\\
\be_r(s-r(2n+1)) & s\in[r(2n+1)-n,(r+1)(2n+1)-n-1]\\
\ga_j(s) & s\in[l(2n+1)-n,k-2n-1]\\
c_{l+1}(s-k+2n+1) & s\in[k-2n-1,k-2n]\\
\be_{x,\tau}(s-k-(M-n)_+) & s\in[k-2n,\tau+k+(M-n)_+]
\end{cases}\]
where $c_r:[0,1]\to\T^d=\R^d/\Z^d$ is defined for
$n$ large by
\[c_r(s)=\begin{cases}
(1-s)\be_{r-1}(s+n)+s\be_r(s-1-n)& r< l\\
(1-s)\be_l(s+n)+s\ga_j(s)&r=l\\
(1-s)\ga_j(s)+s\be_{x,\tau}(s-1-n-M)&r=l+1.
\end{cases}\]
Notice that if $(x,\tau)=\ga_j(\tau)$ then $\af_k(s)=\ga_j(s)$
for $s\ge l(2n+1)-n$.

Since $L\ge 0$, when $n\ge M$ we have
\begin{align*}
A_L(\af_k)&=A_L(\be_0|_{[0,n]})+\sum_{r=1}^{l-1}A_L(\be_r|_{[-n,n]})
+A_L(\be_{x,\tau}|_{[-2n,\tau]})+\sum_{r=1}^{l+1}A_L(c_r)\\
&\le h(\yb_j,\xb_{k_1})+\sum_{r=1}^{l-1}h(\xb_{k_r},\xb_{k_{r+1}})
+h(\xb_j,(x,[\tau]))+\sum_{r=1}^{l+1}A_L(c_r)\\   
&= h(\yb_j,\xb_j)+h(\xb_j,(x,[\tau]))+\sum_{r=1}^{l+1}A_L(c_r)
\end{align*}
Thus
\[\cL_{\tau+k}u(x)-\ub(x,[\tau])\le\sum_{r=1}^{l+1}A_L(c_r)\]
Since $L=0$ on $\tcA$ and there is $C_3\ge C_2$ with
\begin{equation}\label{d(c,ga)}
  d(dc_r(s),d\ga_{k_r}(s)),\; d(dc_{l+1}(s),d\ga_j(s))\le C_3e^{-\lam n}
\end{equation}
we have
\[\cL_{\tau+k}u(x)-\ub(x,[\tau])\le C_4e^{-\lam n}\le
C_5e^{-\lam k/2m}.\]
When $M>n$ define
\begin{align*}
c_{l+2}(s)&=-s\be_{x,\tau}(s+\tau+n-M)+(1+s)\be_{x,\tau}(s+\tau) \quad s\in[-1,0]\\
\hat\af_k(s)&=c_{l+2}(s-k-\tau)\quad s\in[k+\tau-1,k+\tau].
\end{align*}
Since $L\ge 0$ we have as before 
\[A_L(\af_k|[0,k+\tau-1])+A_L(\hat\af_k)\le
h(\yb_j,\xb_j)+h(\xb_j,(x,[\tau]))+\sum_{r=1}^{l+2}A_L(c_r)\] 
Since $L=0$ on $\tcA$, \eqref{d(c,ga)} and
$d(dc_{l+2}(s),d\ga_i(s))\le C_3e^{-\lam n}$ 
we have
\[\cL_{\tau+k}u(x)-\ub(x,[\tau])\le C_4e^{-\lam n}\le
C_5e^{-\lam k/2m}.\]
\begin{remark}
The points $\yb_i$ depend on the function $u$. To get a constant $C_4$ 
independent of $u$ we should manage situations similar to the case $M>0$
above for a continuum of points.
\end{remark}

We now establish the opposite inequality.

For $x\in M$, $t>0$ let  $\ga=\ga_{x,t}:[0,t]\to\T^d$ be such that
$\ga(t)=x$ and
\[\cL_tu(x)=u(\ga(0))+A_L(\ga_t)
=u(\ga(0))+F_{0,t}(\gab(0),x,[t]).\]
For any integer $j\in[0,t]$, $i\in[1,m]$ we have
\begin{align*}
  \ub(x,[t])&\le \ub(\xb_i)+h(\xb_i,x,[t])\\
&\le u(\ga(0))+h(\gab(0),x_i)+h(\xb_i,x,[t]) \\
&\le u(\ga(0))+\Phi(\gab(0),\gab(j))+h(\gab(j),\xb_i)+
h(\xb_i,\gab(j))+\Phi(\gab(j),x,[t])\\ 
&\le u(\ga(0))+A_L(\ga)+h(\gab(j),\xb_i)+h(\xb_i,\gab(j))\\
&=\cL_tu(x)+h(\gab(j),\xb_i)+h(\xb_i,\gab(j))\\
&\le \cL_tu(x)+ K d(\ga(j),x_i).
\end{align*}
Apply Lemma \ref{single} to $V$ given by \eqref{eq:vecindad}.
If $\ga:[0,t]\to\T^d$ is a minimizer and $n+1=\lb\dfrac t{2N}\rb$,
$N=N(V)$, then  $(\ga,\dga,[\cdot])$ stays in one $V_i$ on an interval
$I$ of lenght $2n+1$. 
Let $j\in\N$ be such that $[j,j+2n]\subset I$. Letting
$Y_l=(\ga(j+l),\dga(j+l))$ and writting
\[ g_i(Y_0)=z_-+z_+,\; g_i(Y_{2n})=w_-+w_+,\;z_\pm,w_\pm\in E_i^\pm\] 
we have
\begin{align*}
  g_i(Y_n)&= 
D_i^n(z_-)+ D_i^{-n}(w_+),\\
\|g_i(Y_n)\|&\le e^{-\lam n}(\|z_-\|+\|w_+\|),\\
d(\ga(j+n),x_i)&\le d(Y_n,(x_i,v_i))\le C_5 e^{-\af\lam n}\le C_6 e^{-\af\lam t/2N}.
\end{align*}

\end{document}